\pdfoutput=1
\documentclass[letterpaper]{amsart}
\usepackage[style=alphabetic]{biblatex}
\usepackage{xurl}
\usepackage{hyperref,amssymb,amsthm,xcolor}
\hypersetup{colorlinks=true,citecolor=red}
\usepackage[nameinlink]{cleveref}

\calclayout

\theoremstyle{plain}
\newtheorem{theorem}{Theorem}[section]
\newtheorem{lemma}[theorem]{Lemma}
\newtheorem{corollary}[theorem]{Corollary}
\newtheorem{proposition}[theorem]{Proposition}

\theoremstyle{definition}
\newtheorem{remark}[theorem]{Remark}
\newtheorem{definition}[theorem]{Definition}

\DeclareMathOperator{\Nil}{Nil}
\newcommand{\red}{\mathsf{red}}
\newcommand{\zero}{\mathsf{zero}}
\newcommand{\add}{\mathsf{add}}
\newcommand{\mult}{\mathsf{mult}}
\newcommand{\intro}{\mathsf{intro}}
\newcommand{\semiprime}{\mathsf{semiprime}}

\begin{document}
\title{A constructive counterpart of the subdirect representation theorem for reduced rings}
\author{Ryota Kuroki}
\thanks{Graduate School of Mathematical Sciences, University of Tokyo.\\\indent
\emph{E-mail address:} {\ttfamily{kuroki-ryota128@g.ecc.u-tokyo.ac.jp}}}

\begin{abstract}
We give a constructive counterpart of the theorem of Andrunakievi\v c and Rjabuhin, which states that every reduced ring is a subdirect product of domains.
As an application, we extract a constructive proof of the fact that every ring $A$ satisfying $\forall x\in A.\ x^3=x$ is commutative from a classical proof.
We also prove a similar result for semiprime ideals.
\end{abstract}

\maketitle

\section{Introduction}
In this paper, all rings are unital.
It is well known that every reduced ring can be represented as a subdirect product of domains (\cite{andrunakievich1968}, \cite{Klein}, \cite[Theorem 12.7]{Lam}).
In the commutative case, a constructive counterpart of this theorem is known as formal Nullstellensatz (\cite[Lemma V-3.2]{StoneSpaces}, \cite[Section 4]{SW2021}).
The main theorem of this paper (\cref{subdirect-representation-constructive}) generalizes it to the noncommutative case. We also prove a similar result for semiprime ideals in \cref{semiprime}.

As an application, we use \cref{subdirect-representation-constructive} to extract a constructive proof of the fact that every ring $A$ satisfying $\forall x\in A.\ x^3=x$ is commutative from a classical proof. See \cite[Section 5.7]{coquand} for another method of extraction using topological models.
\section{A constructive counterpart of the subdirect representation theorem for reduced rings}
\begin{definition}
    A ring $A$ is called reduced if $x^2=0$ implies $x=0$ for all $x\in A$. An ideal $I$ of $A$ is called reduced or completely semiprime if $A/I$ is reduced.
\end{definition}
\begin{lemma}[{\cite[Lemma 1.2]{Krempa}}]\label{permutation-lemma}
    Let $I$ be a reduced ideal of $A$, $x_1,\ldots,x_n\in A$, and $\sigma\in S_n$. If $x_1\cdots x_n\in I$, then $x_{\sigma(1)}\cdots x_{\sigma(n)}\in I$.
\end{lemma}
\begin{definition}
    For a subset $U\subseteq A$, we generate an ideal $\Nil U\subseteq A$ by the following constructors (we generate $(-)\in\Nil U$ as an inductive family):
    \begin{alignat*}{2}
        \intro_x&{}:{}&x\in U&\implies x\in\Nil U,\\
        \zero&{}:{}&&\phantom{{}\implies{}}0\in\Nil U,\\
        \add_{x,y}&{}:{}&x,y\in\Nil U&\implies x+y\in\Nil U,\\
        \mult_{z,x,w}&{}:{}&x\in\Nil U&\implies zxw\in\Nil U,\\
        \red_x&{}:{}&x^2\in\Nil U&\implies x\in\Nil U.
    \end{alignat*}
\end{definition}
The ideal $\Nil U\subseteq A$ is the smallest reduced ideal containing $U$. A ring $A$ is reduced if and only if $\Nil0=0$.
\begin{theorem}
    \label{subdirect-representation-constructive}
    Let $U$ be a subset of a ring $A$ and $a,b,x,y\in A$. If $x\in\Nil(U,a)$ and $y\in\Nil(U,b)$, then $xy\in\Nil(U,ab)$. Here $\Nil(U,a)$ denotes $\Nil(U\cup\{a\})$.
\end{theorem}
\begin{proof}
    The proof is by induction on the witnesses $p,q$ of $x\in\Nil(U,a)$, $y\in\Nil(U,b)$.
    \begin{enumerate}
        \item If $p$ and $q$ are the form $\intro_x({-})$ and $\intro_y({-})$ respectively, then $x\in U\cup\{a\}$ and $y\in U\cup\{b\}$. So $xy\in\Nil(U,ab)$.
        \item If $p$ is $\zero$, then $xy=0\in\Nil(U,ab)$.
        \item If $p$ is of the form $\add_{x_1,x_2}({-},{-})$, then we have $x=x_1+x_2$ and $x_1y,x_2y\in\Nil(U,ab)$ by the inductive hypothesis. So $xy=x_1y+x_2y\in\Nil(U,ab)$.
        \item If $p$ is of the form $\mult_{z,x',w}({-})$, then we have $x=zx'w$ and $x'y\in\Nil(U,ab)$ by the inductive hypothesis. So $xy=zx'wy$ is in $\Nil(U,ab)$ by \cref{permutation-lemma}.
        \item If $p$ is of the form $\red_{x}({-})$, then we have $x^2y\in\Nil(U,ab)$ by the inductive hypothesis. So $(xy)^2$ and $xy$ are in $\Nil(U,ab)$ by \cref{permutation-lemma}.
    \end{enumerate}
    The remaining cases can be dealt similarly.
\end{proof}
\begin{corollary}
    \label{corollary-subdirect-representation}
    If $U$ is a subset of a ring $A$ and $a,b\in A$, then $\Nil(U,a)\cap\Nil(U,b)\subseteq\Nil(U,ab)$.
\end{corollary}
\begin{remark}
    The subdirect representation theorem for reduced rings (\cite{andrunakievich1968}) easily follows from \cref{subdirect-representation-constructive} in classical mathematics. Let $U$ be a subset of a ring $A$. We prove that $\Nil U$ is equal to the intersection $\bigcap_{P\supseteq U}P$ of all completely prime ideals containing $U$. We can prove $\Nil U\subseteq\bigcap_{P\supseteq U}P$ since $\Nil U\subseteq A$ is the smallest reduced ideal containing $U$. Conversely, let $x\in A-\Nil U$ and prove that there exists a completely prime ideal $P$ containing $U$ and not containing $x$.
    The set of reduced ideals containing $U$ and not containing $x$ has a maximal element $P$ by Zorn's lemma. Let $a,b\in A$. If $ab\in P$, then $x\notin\Nil(P,a)$ or $x\notin\Nil(P,b)$ by \cref{corollary-subdirect-representation}. So $a\in P$ or $b\in P$ by the maximality of $P$. So $P$ is completely prime.
\end{remark}
\begin{remark}
    The subdirect representation theorem for reduced rings (\cite{andrunakievich1968}) can be regarded as a semantic counterpart of a conservation theorem.
    We generate a single-conclusion entailment relation $\rhd$ and an entailment relation $\vdash$ on a ring $A$ by the following axioms:
    \begin{gather*}
        {}\rhd0,\quad a,b\rhd a+b,\quad a\rhd xay,\quad a^2\rhd a.\\
        {}\vdash0,\quad a,b\vdash a+b,\quad a\vdash xay,\quad ab\vdash a,b,\quad 1\vdash{}.
    \end{gather*}
    If $U$ is a finite subset of $A$ and $a\in A$, then $U\rhd a$ is equivalent to $a\in\Nil U$.
    So $\vdash$ is a conservative extension of $\rhd$ by \cref{corollary-subdirect-representation} and Universal Krull (\cite[Corollary 3]{RSW}).
    It is possible to use \cite[Lemma 4.34]{Wes18} instead of Universal Krull.
    This implies the subdirect representation theorem by a completeness theorem for entailment relations (\cite[Proposition 1.4]{scott1974}).
\end{remark}
We give an application of \cref{corollary-subdirect-representation} to a constructive proof of a commutativity theorem.
\begin{proposition}\label{commutativity-example}
    If $A$ is a ring, $x\in A$, and $c_1,\ldots,c_n\in Z(A):=\{z\in A:\forall y\in A.\ [z,y]=0\}$, then $[x,y]:=xy-yx\in\Nil((x-c_1)\cdots(x-c_n))$ for all $y\in A$. In particular, if $x$ is an element of a reduced ring $A$ such that $x^3=x$, then $x\in Z(A)$.
\end{proposition}
\begin{proof}
    We have $\bigcap_{k=1}^n\Nil(x-c_k)\subseteq\Nil((x-c_1)\cdots(x-c_n))$ by \cref{corollary-subdirect-representation}. The assertion follows, since $[x,y]=[x-c_k,y]\in\Nil(x-c_k)$ for all $k\in\{1,\ldots,n\}$ and $y\in A$.
\end{proof}
This proposition implies that every ring $A$ such that $x^3=x$ for all $x\in A$ is commutative. See \cite{Buckley-MacHale} for other constructive proofs of this fact. Note that rings are not assumed to be unital in the paper, although it does not make much difference (\cite[Proposition 2.14]{brandenburg2023}).
\section{A similar result for semiprime ideals}\label{semiprime}
\begin{definition}
    A ring $A$ is called semiprime if $xAx=\{0\}$ implies $x=0$ for all $x\in A$. An ideal $I$ of $A$ is called semiprime if $A/I$ is semiprime.
\end{definition}
\begin{definition}
    For a subset $U\subseteq A$, we generate an ideal $\sqrt{U}\subseteq A$ by the following constructors:
    \begin{alignat*}{2}
        \intro_x&{}:{}&x\in U&\implies x\in\sqrt U,\\
        \zero&{}:{}&&\phantom{{}\implies{}}0\in\sqrt U,\\
        \add_{x,y}&{}:{}&x,y\in\sqrt U&\implies x+y\in\sqrt U,\\
        \mult_{z,x,w}&{}:{}&x\in\sqrt U&\implies zxw\in\sqrt U,\\
        \semiprime_x&{}:{}&(\forall z\in A.\ xzx\in\sqrt{U})&\implies x\in\sqrt U.
    \end{alignat*}
\end{definition}
The ideal $\sqrt{U}\subseteq A$ is the smallest semiprime ideal containing $U$. A ring $A$ is semiprime if and only if $\sqrt0=0$.
\begin{theorem}
    \label{subdirect-representation-constructive2}
    Let $U$ be a subset of a ring $A$ and $a,b,x,y\in A$. If $x\in\sqrt{U,a}$ and $y\in\sqrt{U,b}$, then $xAy\subseteq\sqrt{U,aAb}$. Here $\sqrt{U,a}$ and $\sqrt{U,aAb}$ denotes $\sqrt{U\cup\{a\}}$ and $\sqrt{U\cup aAb}$, respectively.
\end{theorem}
\begin{proof}
    The proof is by induction on the witnesses $p,q$ of $x\in\sqrt{U,a}$, $y\in\sqrt{U,b}$.
    \begin{enumerate}
        \item If $p$ and $q$ are the form $\intro_x({-})$ and $\intro_y({-})$ respectively, then $x\in U\cup\{a\}$ and $y\in U\cup\{b\}$. So $xAy\subseteq\sqrt{U,aAb}$.
        \item If $p$ is $\zero$, then $xAy=\{0\}\subseteq\sqrt{U,aAb}$.
        \item If $p$ is of the form $\add_{x_1,x_2}({-},{-})$, then we have $x=x_1+x_2$ and $x_1Ay,x_2Ay\subseteq\sqrt{U,aAb}$ by the inductive hypothesis. So $xAy\subseteq x_1Ay+x_2Ay\subseteq\sqrt{U,aAb}$.
        \item If $p$ is of the form $\mult_{z,x',w}({-})$, then we have $x=zx'w$ and $x'Ay\subseteq\sqrt{U,aAb}$ by the inductive hypothesis. So $xAy=zx'wAy\subseteq\sqrt{U,aAb}$.
        \item If $p$ is of the form $\semiprime_{x}({-})$, then we have $xzxAy\subseteq\sqrt{U,aAb}$ for all $z\in A$ by the inductive hypothesis.
        So $xwyAxwy\subseteq\sqrt{U,aAb}$ for all $w\in A$. So $xAy\subseteq\sqrt{U,aAb}$.
    \end{enumerate}
    The remaining cases can be dealt similarly.
\end{proof}
\begin{corollary}
    If $U$ is a subset of a ring $A$ and $a,b\in A$, then $\sqrt{U,a}\cap\sqrt{U,b}\subseteq\sqrt{U,aAb}$.
\end{corollary}
\section*{Acknowledgements}
I would like to express my deepest gratitude to my supervisor, Ryu Hasegawa, for his support.
I would like to thank Daniel Misselbeck-Wessel for his helpful advice.

This research was supported by Forefront Physics and Mathematics Program to Drive Transformation (FoPM), a World-leading Innovative Graduate Study (WINGS) Program, the University of Tokyo.
\printbibliography
\end{document}